\documentclass[a4paper,12pt]{article}

\usepackage[utf8]{inputenc}
\usepackage[english]{babel}
\usepackage{amsmath, amssymb, amsthm}
\usepackage{enumerate}
\usepackage[colorlinks=true,linkcolor=blue,citecolor=blue]{hyperref}

\usepackage{tikz}
\usetikzlibrary{fadings}
\tikzfading[name=fade out,
            inner color=transparent!0,
            outer color=transparent!100]
\usetikzlibrary{calc,arrows,decorations.pathreplacing,fadings,3d,positioning}

\title{Palindromes and periodic continued fractions.
       \thanks{ This research was supported by
                RFBR grant 15-01-05700,
                and ``Dynasty'' Foundation}}
\author{Oleg\,N.\,German, Ibragim\,A.\,Tlyustangelov}
\date{}

\theoremstyle{definition}
\newtheorem{definition}{Definition}

\newtheorem*{notation*}{Notation}

\theoremstyle{remark}
\newtheorem{remark}{Remark}
\newtheorem*{remark*}{Remark}

\theoremstyle{plain}
\newtheorem{theorem}{Theorem}
\newtheorem{lemma}{Lemma}

\newtheorem{proposition}{Proposition}
\newtheorem{corollary}{Corollary}
\newtheorem*{statement*}{Statement}
\newtheorem*{corollary*}{Corollary}

\DeclareMathOperator{\conv}{conv}

\renewcommand{\phi}{\varphi}

\renewcommand{\vec}[1]{\mathbf{#1}}
\renewcommand{\geq}{\geqslant}

\newcommand{\R}{\mathbb{R}}
\newcommand{\Z}{\mathbb{Z}}
\newcommand{\Q}{\mathbb{Q}}

\newcommand{\cC}{\mathcal{C}}

\newcommand{\cK}{\mathcal{K}}
\newcommand{\cL}{\mathcal{L}}

\newcommand{\cP}{\mathcal{P}}

\newcommand{\Sl}{\textup{SL}}
\newcommand{\Gl}{\textup{GL}}

\newcommand{\tr}[1]{{#1}^\intercal}

\begin{document}

\maketitle

\begin{abstract}
  This paper is devoted to a detailed exposition of geometry of continued fractions. We pay particular interest to the case of quadratic irrationalities and use the technique described to prove a criterion for the continued fraction of a quadratic surd to have a symmetric period.
\end{abstract}


\section{Introduction}

Since the times of Legendre \cite{legendre} it has been well known that for each rational $r>1$ different from a perfect square we have
\begin{equation} \label{eq:sqrt_expansion}
  \sqrt r=[a_0;\overline{a_1,a_2,\ldots,a_2,a_1,2a_0}].
\end{equation}
Particularly, a period of this continued fraction read back to front is again a period. Lately we have been witnessing attempts to find a criterion for a quadratic irrationality to have this kind of period symmetry (see \cite{arnold_UMN}, \cite{aicardi}, and also \cite{burger}). However, the corresponding criterion has been known for almost a century, though apparently it has never been formulated in a straightforward way. We decided to use these circumstances to demonstrate how geometry of numbers can be applied to prove such statements. To this end we describe in detail the geometric approach to continued fractions with an emphasis on quadratic irrationalities (see also \cite{korkina_2dim}, \cite{karpenkov_book}, \cite{erdos_gruber_hammer}).

First of all, let us agree on terminology. We use the word \emph{period} to denote both a repeating finite sequence of elements of a periodic sequence and the whole family of such sequences.

\begin{definition}
  We say that a finite sequence $(a_1,a_2,\ldots,a_{t-1},a_t)$ is

  a) a \emph{regular palindrome} if $a_k=a_{t+1-k}$ for each $k\in\{1,\ldots,t\}$;

  b) a \emph{cyclic palindrome} if there is a cyclic shift $\sigma$ of indices, such that $a_k=a_{\sigma(t+1-k)}$ for each $k\in\{1,\ldots,t\}$.
\end{definition}

If a finite sequence is a cyclic palindrome, then so is its image under any cyclic shift. Thus, if at least one of the `words' representing the period of a periodic sequence is a cyclic palindrome, then so are all such `words', so it is correct to talk about \emph{cyclic palindromic periods}.

\begin{definition}
  We say that the period of a periodic sequence is \emph{cyclic palindromic} if the `words' representing the period are cyclic palindromes.
\end{definition}

Clearly, a period read back to front is again a period of the same periodic sequence if and only if this period is cyclic palindromic. So, the initial question is actually the question of finding a `nice' criterion for a continued fraction to have a cyclic palindromic period. Notice that a similar question concerning regular palindromes differs from ours. Indeed, not every cyclic palindrome can be turned into a regular one by a cyclic shift: $(1,2)$ and $(1,1,1,2)$ are examples of such sequences. Moreover, the answer to the question concerning regular palindromes follows easily from the classical Galois theorem on reduced quadratic irrationalities. It was his first paper \cite{galois} where he proved particularly that if
\[ \alpha=[\overline{a_0;a_1,\ldots,a_t}], \]
then for the conjugate $\bar\alpha$ of $\alpha$ we have
\[ -1/\bar\alpha=[\overline{a_t;a_{t-1},\ldots,a_0}]. \]
Hence a criterion follows immediately:

\begin{proposition} \label{prop:regular_palindrome_criterion}
  Let $\alpha$ be a quadratic irrationality. Then its continued fraction has a period which is a regular palindrome if and only if $\alpha\sim\omega$: $\omega\bar\omega=-1$.
\end{proposition}

Here and below $\bar\omega$ denotes the conjugate of $\omega$, and the equivalence $\alpha\sim\omega$ means (cf. \cite{khintchine_CF}, \cite{perron_book}) that the continued fractions of $\alpha$ and $\omega$ have same `tails'. An equivalent description is provided by Serret's theorem \cite{serret}, which states that $\alpha\sim\omega$ if and only if there are $a,b,c,d\in\Z$, such that $ac-bd=\pm1$ and
\[ \alpha=\frac{a\omega+b}{c\omega+d}\,. \]

Given a cyclic palindrome, it can be cyclicly shifted so that it becomes either a regular palindrome or a regular palindrome with an extra symbol attached to it. In case this extra symbol is an integer, it is either even or odd. The following two statements with very simple proofs involving the mentioned above Galois theorem can be found in Perron's book \cite{perron_book} (Satz 3.9 and Satz 3.30, see also \cite{venkov}).

\begin{proposition}[Legendre \cite{legendre}, Perron \cite{perron_book}] \label{prop:cyclic_palindrome_criterion_even_extra}
  Let $\alpha$ be a quadratic irrationality. Then its continued fraction has a period which is a regular palindrome with an extra even partial quotient attached if and only if $\alpha\sim\sqrt r$, $r\in\Q$.
\end{proposition}

\begin{proposition}[Kraitchik \cite{kraitchik}, Perron \cite{perron_book}] \label{prop:cyclic_palindrome_criterion_odd_extra}
  Let $\alpha$ be a quadratic irrationality. Then its continued fraction has a period which is a regular palindrome with an extra odd partial quotient attached if and only if $\alpha\sim1/2+\sqrt r$, $r\in\Q$.
\end{proposition}

Now, Propositions \ref{prop:regular_palindrome_criterion}, \ref{prop:cyclic_palindrome_criterion_even_extra}, \ref{prop:cyclic_palindrome_criterion_odd_extra} together give us the desired criterion. We formulate it as follows, with a little symmetrizing addition (statement \textup{(c)}).

\begin{theorem} \label{t:cyclic_palindrome_criterion}
  The continued fraction of a quadratic irrationality $\alpha$ has a cyclic palindromic period if and only if one of the following statements holds:

  \vskip 0.3mm
  \hskip -2.5mm
  \begingroup
  \addtolength{\jot}{-0.25em}
  $ \begin{aligned}
      \textup{(a)}\hskip 2.50mm & \alpha\sim\omega:\ \omega+\bar\omega=0\ (\text{i.e. }\omega^2\in\Q); \\
      \textup{(b)}\hskip 2.37mm & \alpha\sim\omega:\ \omega+\bar\omega=1\ (\text{i.e. }(\omega-1/2)^2\in\Q); \\
      \textup{(c)}\hskip 2.58mm & \alpha\sim\omega:\ \omega\bar\omega=1; \\
      \textup{(d)}\hskip 2.37mm & \alpha\sim\omega:\ \omega\bar\omega=-1.
    \end{aligned} $
  \endgroup

  \noindent
  Besides that, $(b)$ is equivalent to $(c)$.
\end{theorem}

In Section \ref{sec:proof}, after having described in Section \ref{sec:geometry_of_CF} geometry of continued fractions of quadratic irrationalities, we shall prove Theorem \ref{t:cyclic_palindrome_criterion} geometrically. We notice that the only thing in Theorem \ref{t:cyclic_palindrome_criterion} which does not follow directly from Propositions \ref{prop:regular_palindrome_criterion}, \ref{prop:cyclic_palindrome_criterion_even_extra}, \ref{prop:cyclic_palindrome_criterion_odd_extra} is the equivalence of statements \textup{(b)} and \textup{(c)}.

Before finishing this Introduction we would also like to notice that it follows from Propositions \ref{prop:cyclic_palindrome_criterion_even_extra} and \ref{prop:cyclic_palindrome_criterion_odd_extra} (or statements \textup{(a)} and \textup{(b)} of Theorem \ref{t:cyclic_palindrome_criterion}) that all quadratic integers have cyclic palindromic periods. One may ask if the periods of quadratic units can be described more explicitly. Such a description is provided by the following simple statement.

\begin{proposition} \label{prop:unit_period_description}
  For each positive integer $q$ we have

  \vskip 2mm
  \textup{(a)}\,\ $\alpha^2-q\alpha-1=0\iff\alpha=q+\dfrac1\alpha=[\overline q]$;

  \textup{(b)}\,\ $\alpha^2-(q+2)\alpha+1=0\iff\alpha-1=q+\cfrac{1}{1+\cfrac1{\alpha-1}}=[\overline{q;1}]$.
\end{proposition}

Despite its simplicity we could not find Proposition \ref{prop:unit_period_description} in the literature, so we would be glad to find out who was the first to discover it.

\section{Geometry of continued fractions} \label{sec:geometry_of_CF}

Given a real $\alpha$ we shall denote by $\cL_\alpha$ the line in $\R^2$ which passes through the points $(0,0)$ and $(1,\alpha)$. Clearly, there are no nonzero integer points on $\cL_\alpha$ if and only if $\alpha$ is irrational.

Throughout this paper we shall consider only irrational $\alpha$ in order to avoid the necessity to take into account what happens when $\cL_\alpha$ meets an integer point.

\subsection{Klein polygons of adjacent angles}

First, let us consider the following general construction. Let $\alpha$ and $\beta$ be distinct (irrational) real numbers. Then $\cL_\alpha$ and $\cL_\beta$ split the plane into four angles. For each of those angles let us consider the convex hull of nonzero integer points contained in it. The four unbounded (generalised) polygons thus obtained are called \emph{Klein polygons}\footnote{Sometimes the term Klein polygon is applied to the \emph{boundary} of the convex hull, which is a broken line with vertices in $\Z^2$. We prefer to call this boundary a \emph{sail}, following Skubenko, Arnold et al.}. We say that two Klein polygons are \emph{adjacent} if they correspond to adjacent angles. The vertices of Klein polygons all belong to $\Z^2$, therefore, we can talk about \emph{integer lengths} of their edges and \emph{integer angles} between them.

\begin{definition}
  A line segment is said to be \emph{integer} if its endpoints are in $\Z^2$.
  An integer segment is called \emph{empty} if it contains no integer points other than its endpoints.
  The \emph{integer length} of an integer segment is defined as the number of empty segments contained in it.
\end{definition}

\begin{definition} \label{def:integer_angle}
  Given two empty integer segments with a common endpoint, the area of the parallelogram spanned by them is called the \emph{integer angle} between those segments. Given two arbitrary integer segments with a common endpoint, the \emph{integer angle} between them is defined as the integer angle between their empty subsegments incident to their common endpoint.
\end{definition}

Generally, the set of integer points contained in the parallelogram mentioned in Definition \ref{def:integer_angle} may have a rather chaotic structure. It appears that in the case of a Klein polygon's adjacent edges all those points lie on a fixed diagonal of the parallelogram.

\begin{proposition} \label{prop:vertex_sprout}
  Let $\vec v$ be a vertex of a Klein polygon $\cK$ and let $\vec u$ and $\vec w$ be the closest to $\vec v$ integer points on the edges of $\cK$ incident to $\vec v$ (see Fig. \ref{fig:vertex_sprout}). Denote by $\cP_{\vec v}$ the parallelogram determined by $\vec u$, $\vec v$, $\vec w$. Then all the integer points in $\cP_{\vec v}$ different from $\vec u$ and $\vec w$ are positive integer multiples of $\vec v$. Particularly, they all lie on the diagonal of $\cP_{\vec v}$ incident to $\vec v$.
\end{proposition}

\begin{proof}
  Since $\vec u$, $\vec v$, $\vec w$ lie on edges of $\cK$, the triangles $\vec 0\vec v\vec u$ and $\vec 0\vec v\vec w$ are empty (i.e. they contain no integer points other than the vertices). This means that both pairs $\{\vec v,\vec u\}$ and $\{\vec v,\vec w\}$ are bases of $\Z^2$. Hence $\vec u$ and $\vec w$ lie at the same distance from the line generated by $\vec v$, their sum $\vec u+\vec w$ belongs to it, and all the integer points closer to that line than $\vec u$ and $\vec w$ are integer multiples of $\vec v$.
\end{proof}

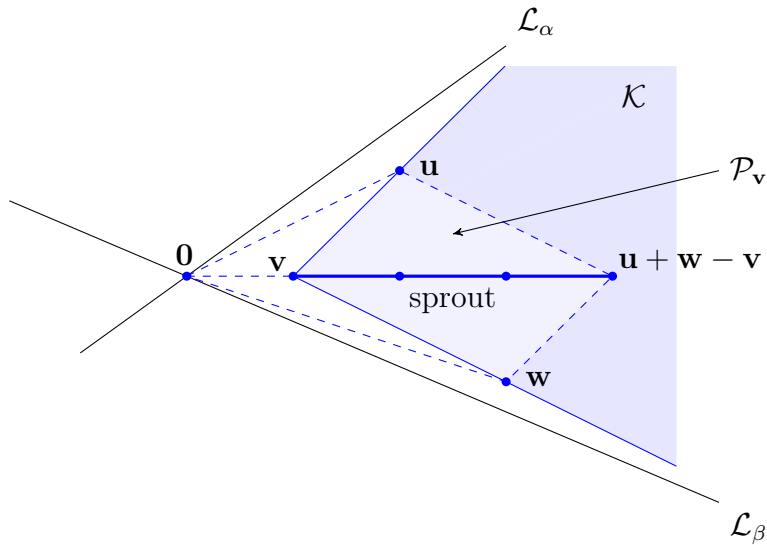
\begin{figure}[h]
  \centering
  \begin{tikzpicture}[scale=1.4]
    \draw[color=black] plot[domain=-1:3] (\x, {8*\x/11}) node[above right]{$\cL_\alpha$};
    \draw[color=black] plot[domain=-5/3:5] (\x, {-3*\x/7}) node[below right]{$\cL_\beta$};

    \fill[blue!10!,path fading=east]
        (4.6,1.99) -- (1,0) -- (4.6,-1.8) -- cycle;
    \fill[blue!10!,path fading=north]
        (2.99,1.99) -- (1,0) -- (4.6,1.99) -- cycle;
    \fill[blue!5!]
        (1,0) -- (3,-1) -- (4,0) -- (2,1) -- cycle;

    \draw[color=blue] (2.99,1.99) -- (1,0) -- (4.6,-1.8);
    \draw[very thick,color=blue] (1,0) -- (4,0);

    \draw[dashed,color=blue] (0,0) -- (1,0);
    \draw[dashed,color=blue] (0,0) -- (2,1) -- (4,0) -- (3,-1) -- cycle;

    \node[fill=blue,circle,inner sep=1.2pt] at (0,0) {};
    \node[fill=blue,circle,inner sep=1.2pt] at (1,0) {};
    \node[fill=blue,circle,inner sep=1.2pt] at (2,0) {};
    \node[fill=blue,circle,inner sep=1.2pt] at (3,0) {};
    \node[fill=blue,circle,inner sep=1.2pt] at (4,0) {};

    \node[fill=blue,circle,inner sep=1.2pt] at (2,1) {};

    \node[fill=blue,circle,inner sep=1.2pt] at (3,-1) {};

    \draw (-0.02,0) node[above]{$\vec 0$};
    \draw (1.05,-0.05) node[above left]{$\vec v$};
    \draw (2.08,1.03) node[right]{$\vec u$};
    \draw (3.08,-1+0.03) node[right]{$\vec w$};
    \draw (4-0.05,-0.05) node[above right]{$\vec u+\vec w-\vec v$};

    \draw (4.2,1.5) node[above]{$\cK$};
    \draw[->,>=stealth',color=black,thin] (5,1) node[right]{$\cP_{\vec v}$} -- (2.5,0.4);
    \draw (2.5,0) node[below]{sprout};
  \end{tikzpicture}
  \caption{Vertex sprout} \label{fig:vertex_sprout}
\end{figure}

\begin{definition} \label{def:vertex_sprout}
  Let $\cK$, $\vec v$ and $\cP_{\vec v}$ be as in Proposition \ref{prop:vertex_sprout}. We call the diagonal of $\cP_{\vec v}$ incident to $\vec v$ a \emph{vertex sprout}.
\end{definition}

Thus, the integer angle between two adjacent edges of a Klein polygon equals the integer length of the corresponding vertex sprout. There is a nice correspondence between the edges of a given Klein polygon and the vertex sprouts of the adjacent one (see also \cite{korkina_2dim}).

\begin{proposition} \label{prop:edge_vs_sprout}
  Let $\cK_1$ and $\cK_2$ be two adjacent Klein polygons. Let $\vec E_1$ denote the set of all the edges of $\cK_1$ and let $\vec S_1$ denote the set of all the vertex sprouts of $\cK_1$. Let $\vec E_2$ and $\vec S_2$ denote the same for $\cK_2$. Then there is a bijection $\phi:\vec E_1\cup\vec S_1\to\vec E_2\cup\vec S_2$ such that

  \textup{(a)} $\phi(\vec E_1)=\vec S_2$, $\phi(\vec S_1)=\vec E_2$;

  \textup{(b)} $\phi$ preserves integer lengths;

  \textup{(c)} each element of $\vec E_1\cup\vec S_1$ is parallel to its image under $\phi$;

  \textup{(d)} an edge and a sprout have a common endpoint whenever so do their images under $\phi$.
\end{proposition}

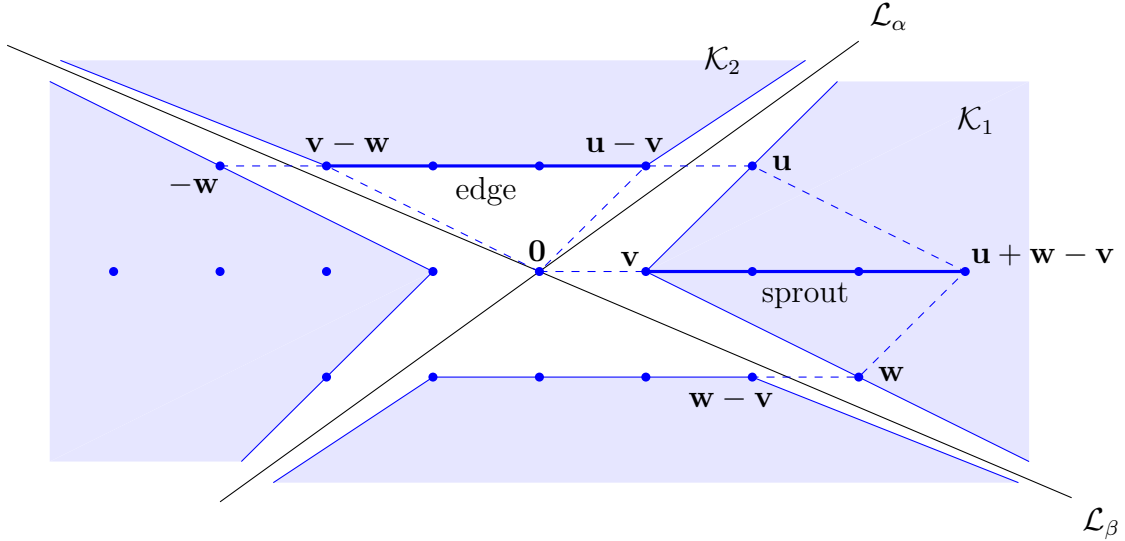
\begin{figure}[h]
  \centering
  \begin{tikzpicture}[scale=1.4]
    \draw[color=black] plot[domain=-3:3] (\x, {8*\x/11}) node[above right]{$\cL_\alpha$};
    \draw[color=black] plot[domain=-5:5] (\x, {-3*\x/7}) node[below right]{$\cL_\beta$};

    \fill[blue!10!,path fading=east]
        (4.6,1.8) -- (1,0) -- (4.6,-1.8) -- cycle;
    \fill[blue!10!,path fading=north]
        (2.8,1.8) -- (1,0) -- (4.6,1.8) -- cycle;
    \fill[blue!10!,path fading=west]
        (-4.6,-1.8) -- (-1,0) -- (-4.6,1.8) -- cycle;
    \fill[blue!10!,path fading=south]
        (-2.8,-1.8) -- (-1,0) -- (-4.6,-1.8) -- cycle;
    \fill[blue!10!,path fading=north]
        (-4.5,2) -- (-2,1) -- (1,1) -- (2.5,2) -- cycle;
    \fill[blue!10!,path fading=south]
        (4.5,-2) -- (2,-1) -- (-1,-1) -- (-2.5,-2) -- cycle;

    \draw[color=blue] (2.8,1.8) -- (1,0) -- (4.6,-1.8);
    \draw[color=blue] (-2.8,-1.8) -- (-1,0) -- (-4.6,1.8);
    \draw[color=blue] (-4.5,2) -- (-2,1) -- (1,1) -- (2.5,2);
    \draw[color=blue] (4.5,-2) -- (2,-1) -- (-1,-1) -- (-2.5,-2);

    \draw[very thick,color=blue] (-2,1) -- (1,1);
    \draw[very thick,color=blue] (1,0) -- (4,0);

    \draw[dashed,color=blue] (-3,1) -- (-2,1) -- (0,0) -- (1,1) -- (2,1) -- (4,0) -- (3,-1) -- (2,-1);
    \draw[dashed,color=blue] (0,0) -- (1,0);

    \node[fill=blue,circle,inner sep=1.2pt] at (-4,0) {};
    \node[fill=blue,circle,inner sep=1.2pt] at (-3,0) {};
    \node[fill=blue,circle,inner sep=1.2pt] at (-2,0) {};
    \node[fill=blue,circle,inner sep=1.2pt] at (-1,0) {};
    \node[fill=blue,circle,inner sep=1.2pt] at (0,0) {};
    \node[fill=blue,circle,inner sep=1.2pt] at (1,0) {};
    \node[fill=blue,circle,inner sep=1.2pt] at (2,0) {};
    \node[fill=blue,circle,inner sep=1.2pt] at (3,0) {};
    \node[fill=blue,circle,inner sep=1.2pt] at (4,0) {};

    \node[fill=blue,circle,inner sep=1.2pt] at (-3,1) {};
    \node[fill=blue,circle,inner sep=1.2pt] at (-2,1) {};
    \node[fill=blue,circle,inner sep=1.2pt] at (-1,1) {};
    \node[fill=blue,circle,inner sep=1.2pt] at (0,1) {};
    \node[fill=blue,circle,inner sep=1.2pt] at (1,1) {};
    \node[fill=blue,circle,inner sep=1.2pt] at (2,1) {};

    \node[fill=blue,circle,inner sep=1.2pt] at (-2,-1) {};
    \node[fill=blue,circle,inner sep=1.2pt] at (-1,-1) {};
    \node[fill=blue,circle,inner sep=1.2pt] at (0,-1) {};
    \node[fill=blue,circle,inner sep=1.2pt] at (1,-1) {};
    \node[fill=blue,circle,inner sep=1.2pt] at (2,-1) {};
    \node[fill=blue,circle,inner sep=1.2pt] at (3,-1) {};

    \draw (-0.02,0) node[above]{$\vec 0$};
    \draw (1.05,-0.05) node[above left]{$\vec v$};
    \draw (2.08,1.03) node[right]{$\vec u$};
    \draw (3.08,-1+0.03) node[right]{$\vec w$};
    \draw (4-0.05,-0.05) node[above right]{$\vec u+\vec w-\vec v$};
    \draw (-3+0.1,1.02) node[below left]{$-\vec w$};
    \draw (-2+0.2,1) node[above]{$\vec v-\vec w$};
    \draw (1-0.2,1) node[above]{$\vec u-\vec v$};
    \draw (2-0.2,-1) node[below]{$\vec w-\vec v$};

    \draw (4.1,1.2) node[above]{$\cK_1$};
    \draw (2,2) node[left]{$\cK_2$};

    \draw (-0.5,1) node[below]{edge};
    \draw (2.5,0) node[below]{sprout};
  \end{tikzpicture}
  \caption{Edge--sprout correspondence} \label{fig:edge_vs_sprout}
\end{figure}

\begin{proof}
  Let us denote by $\cC_1$ the angle to which $\cK_1$ corresponds, and by $\cC_2$ --- the one to which $\cK_2$ corresponds. Let $\vec v$ be a vertex of $\cK_1$ and let $\vec u$ and $\vec w$ be as in Proposition \ref{prop:vertex_sprout}. Then, the points $\vec u-\vec v$ and $\vec w-\vec v$ do not belong to $\cC_1$, but one of them belongs to $\cC_2$ and the other one to $-\cC_2$. We may assume that $\vec u-\vec v\in\cC_2$. Then, the segment $[\vec v-\vec w,\vec u-\vec v]$ is an edge of $\cK_2$, since $-\vec w$ and $\vec u$ are not in $\cC_2$, $\vec v$ is primitive, and there are no integer points between the line containing this segment and the line generated by $\vec v$ (see Fig. \ref{fig:edge_vs_sprout}). Obviously, the image of this edge under parallel translation by $\vec w$ is exactly the sprout $[\vec v,\vec u+\vec w-\vec v]$.

  Thus, given a vertex $\vec v$ of $\cK_1$, there is exactly one edge of $\cK_2$ parallel to $\vec v$ and having integer length equal to that of $\vec v$'s sprout. Consider the endpoints of this edge. As we have shown, those are $\vec v-\vec w$ and $\vec u-\vec v$. For each of them there is exactly one edge of $\cK_1$ parallel to it and incident to $\vec v$. For $\vec v-\vec w$ it is the edge starting with the segment $[\vec v,\vec w]$, and for $\vec u-\vec v$ --- the one starting with $[\vec v,\vec u]$.

  Continuing this argument in both directions, we get the desired correspondence.
\end{proof}

Notice that each vertex sprout has a `root' at the vertex and a `top' at its other endpoint, and the `root' is always closer to $\cL_\beta$ and to $\cL_\alpha$ than the `top'. This generates an orientation of all the edges, so that the `beginning' of an edge is closer to $\cL_\beta$ and farther from $\cL_\alpha$ than its `end' (if $\cK_1$ and $\cK_2$ are in same half-plane w.r.t. $\cL_\beta$ and are separated by $\cL_\alpha$, as in Fig. \ref{fig:edge_vs_sprout}). Which allows enumerating all the vertices the following way.

\begin{corollary} \label{cor:naming_vertices}
  Let $\cK_1$ and $\cK_2$ be two adjacent Klein polygons separated by $\cL_\alpha$. We can denote all the vertices of $\cK_1$ and $\cK_2$ as $\vec v_k$, $k\in\Z$, so that for each integer $k$

  \textup{(a)} $\vec v_{k-1}$, $\vec v_k$ form a basis of $\Z^2$;

  \textup{(b)} $[\vec v_{k-2},\vec v_k]$ is an edge of $\cK_1$ if $k$ is even, and of $\cK_2$ if $k$ is odd;

  \textup{\hskip0.23mm(c)\hskip0.23mm} 
               $\vec v_k$ is closer to $\cL_\alpha$ than $\vec v_{k-2}$;

  \textup{(d)} $\vec v_k=\vec v_{k-2}+a_k\vec v_{k-1}$, \\
  where $a_k$ equals both the integer length of $[\vec v_{k-2},\vec v_k]$ and the integer angle between the edges incident to $\vec v_{k-1}$.

  This numeration of vertices is unique up to the choice of the initial vertex.
\end{corollary}

Thus, we have the sequence $(a_k)_{k\in\Z}$ written twice along the boundaries of $\cK_1$ and $\cK_2$ (cf. Fig. \ref{fig:KP_and_CF} below).

\subsection{Korkina's lemma}

Given a basis $\vec v_{-2}$, $\vec v_{-1}$ of $\Z^2$ and a sequence $(a_k)_{k\in\Z}$ of positive integers the recurrence relation
\begin{equation} \label{eq:recurrence_relation}
  \vec v_k=\vec v_{k-2}+a_k\vec v_{k-1}
\end{equation}
determines the whole sequence $(\vec v_k)_{k\in\Z}$. 

\begin{proposition}[Korkina \cite{korkina_2dim}] \label{prop:korkina}
  Let $(a_k)_{k\in\Z}$ be an arbitrary sequence of positive integers and let $[\vec v_{-2},\vec v_0]$ be an integer segment of integer length $a_0$. Suppose that all the integer points that are closer to the line through $\vec v_{-2}$ and $\vec v_0$ than the origin belong to that line. Then there is a unique Klein polygon $\cK$ with vertices $\vec v_{2m}$, $m\in\Z$, such that for each integer $m$

  \textup{(a)} $[\vec v_{2m-2},\vec v_{2m}]$ is an edge of $\cK$;

  \textup{(b)} $a_{2m}$ equals the integer length of $[\vec v_{2m-2},\vec v_{2m}]$;

  \textup{\hskip0.23mm(c)\hskip0.23mm} $a_{2m+1}$ equals the integer angle at $\vec v_{2m}$.
\end{proposition}

\begin{proof}
  If such a $\cK$ exists, then by Corollary \ref{cor:naming_vertices} its vertices and the vertices of an adjacent Klein polygon should satisfy \eqref{eq:recurrence_relation}. This compels us to set $\vec v_{-1}=(\vec v_0-\vec v_{-2})/a_0$ and define the sequence $(\vec v_k)_{k\in\Z}$ by \eqref{eq:recurrence_relation}. Let us also denote
  \[ \Delta_k=\conv(\vec 0,\vec v_{k-2},\vec v_k). \]
  Vectors $\vec v_{-2},\vec v_{-1}$ form a basis of $\Z^2$. We may assume without loss of generality that $\det(\vec v_{-2},\vec v_{-1})=1$. Then it follows from \eqref{eq:recurrence_relation} that for each integer $k$ we have
  \[ \det(\vec v_{k-1},\vec v_k)=(-1)^{k-1}\quad\text{ and }\quad\det(\vec v_{k-2},\vec v_k)=(-1)^ka_k. \]
  The latter equality implies that every two triangles $\Delta_k$ and $\Delta_{k+2}$ share a common side and do not overlap.
  Moreover,
  \[ \det(\vec v_{k-2}-\vec v_k,\vec v_{k+2}-\vec v_k)=-a_ka_{k+2}\det(\vec v_{k-1},\vec v_{k+1})=(-1)^ka_ka_{k+1}a_{k+2}. \]
  Hence it follows that each quadrilateral $\Delta_k\cup\Delta_{k+2}$ is not convex, i.e. the broken line with vertices $\ldots,\vec v_{-2},\vec v_0,\vec v_2,\ldots$ bounds a convex region, a (generalised) convex polygon $\cK$, which borders every $\Delta_{2m}$, but does not overlap any of them. Thus, the union
  \[ \cK\cup\bigcup_{m\in\Z}\Delta_{2m} \]
  is an angle $\cC$ formed by some $\cL_\alpha$ and $\cL_\beta$. It remains to notice that all the nonzero integer points contained in $\Delta_k$ belong to its side $[\vec v_{k-2},\vec v_k]$, so that
  \[ \cK=\conv(\cC\cap\Z^2\backslash\{\vec 0\}). \]
\end{proof}

For any two segments $[\vec v_{-2},\vec v_0]$ and $[\vec v'_{-2},\vec v'_0]$ satisfying the hypothesis of Proposition \ref{prop:korkina} there is a unique operator $A\in\Gl_2(\Z)$ such that $A\vec v_{-2}=\vec v'_{-2}$ and $A\vec v_0=\vec v'_0$. Applying Proposition \ref{prop:korkina} we get the following useful statement.

\begin{corollary} \label{cor:korkina}
  Given two Klein polygons, suppose their 1-skeletons equipped with integer lengths of edges and integer angles at vertices are isomorphic. Then there is an operator in $\Gl_2(\Z)$ which maps one Klein polygon onto the other and respects this isomorphism.
\end{corollary}

\subsection{Klein polygons and continued fractions} \label{sec:KP_and_CF}

\begin{proposition} \label{prop:deep_recursion}
  Within the notation of Corollary \ref{cor:naming_vertices} fix $k,m\in\Z$, $k\geq m$, and define $p$ and $q$ by $\vec v_k=q\vec v_{m-2}+p\vec v_{m-1}$. Then $p$ and $q$ are coprime integers and
  \[ \frac pq=[a_m;a_{m+1},\ldots,a_k]. \]
\end{proposition}

\begin{proof}
  Coprimality follows immediately from the fact that $\vec v_k$ is primitive. The rest is proven by induction on $m$ while $k$ is fixed. The base case $m=k$ repeats statement \textup{(d)} of Corollary \ref{cor:naming_vertices}. The inductive step from $m$ to $m-1$ is provided by the relation
  \begin{multline*}
    \vec v_k=q\vec v_{m-2}+p\vec v_{m-1}= \\
    =q\vec v_{m-2}+p(\vec v_{m-3}+a_{m-1}\vec v_{m-2})=\qquad\qquad \\
    =p\big(\vec v_{m-3}+(a_{m-1}+q/p)\vec v_{m-2}\big).
  \end{multline*}
\end{proof}

Proposition \ref{prop:deep_recursion} already allows us to call the boundaries of two adjacent Klein polygons a \emph{geometric continued fraction} (cf. \cite{karpenkov_book}). But let us show how the sequence $(a_k)_{k\in\Z}$ of integer lengths and angles written along those boundaries is connected to the sequences of partial quotients of $\alpha$ and $\beta$. Most explicitly it is observed if
\begin{equation} \label{eq:nice_alpha_beta}
  \alpha>1,\ \ -1<\beta<0.
\end{equation}
In this case the points $(1,0)$ and $(0,1)$ are vertices of $\cK_1$ and $\cK_2$ (see Fig. \ref{fig:KP_and_CF}) and we can set
\begin{equation} \label{eq:initiation}
  \vec v_{-2}=(1,0),\ \ \vec v_{-1}=(0,1).
\end{equation}

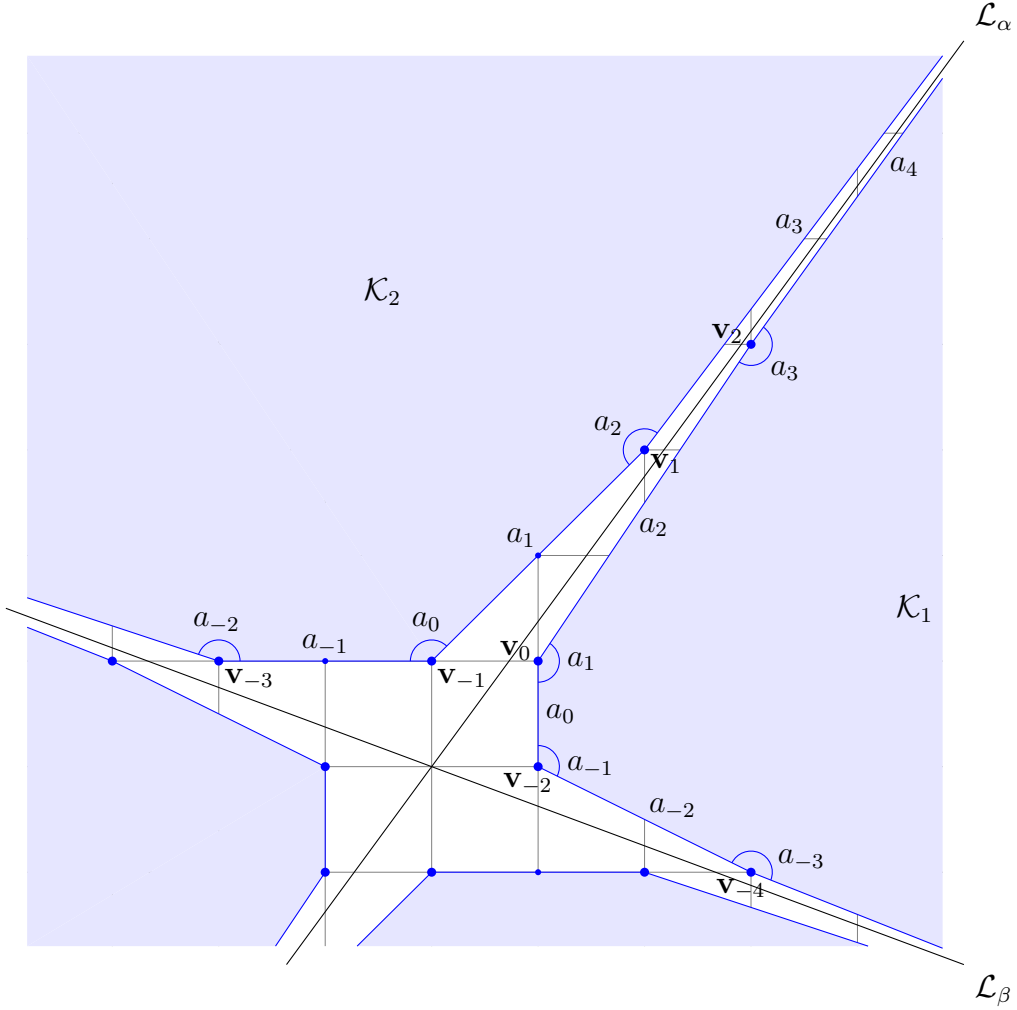
\begin{figure}[h]
  \centering
  \begin{tikzpicture}[scale=1.4]
    \draw[very thin,color=gray,scale=1] (-3.8,-1.7) grid (4.8,55/8-0.2);


    \draw[color=black] plot[domain=-15/11:5] (\x, {11*\x/8}) node[above right]{$\cL_\alpha$};
    \draw[color=black] plot[domain=-4:5] (\x, {-3*\x/8}) node[below right]{$\cL_\beta$};


    \fill[blue!10!,path fading=east]
        (3+1.8,4+1.8*7/5) -- (3,4) -- (1,1) -- (1,0) -- (3,-1) -- (3+1.8,-1-1.8*2/5) -- cycle;
    \fill[blue!10!,path fading=north]
        (2+2.8,3+2.8*4/3) -- (2,3) -- (0,1) -- (-2-1.8,3+2.8*4/3) -- cycle;
    \fill[blue!10!,path fading=west]
        (-2-1.8,3+2.8*4/3) -- (0,1) -- (-2,1) -- (-2-1.8,1+1.8/3) -- cycle;
    \fill[blue!10!,path fading=south]
        (-0.7,-1.7) -- (0,-1) -- (2,-1) -- (4.1,-1.7) -- cycle;
    \fill[blue!10!,path fading=west]
        (-3.8,-1.7) -- (-1,0) -- (-3,1) -- (-3.8,1+0.8*2/5) -- cycle;
    \fill[blue!10!,path fading=south]
        (-1-0.7*2/3,-1.7) -- (-1,-1) -- (-1,0) -- (-3.8,-1.7) -- cycle;

    \draw[color=blue] (3+1.8,4+1.8*7/5) -- (3,4) -- (1,1) -- (1,0) -- (3,-1) -- (3+1.8,-1-1.8*2/5);
    \draw[color=blue] (2+2.8,3+2.8*4/3) -- (2,3) -- (0,1) -- (-2,1) -- (-2-1.8,1+1.8/3);
    \draw[color=blue] (-0.7,-1.7) -- (0,-1) -- (2,-1) -- (4.1,-1.7);
    \draw[color=blue] (-1-0.7*2/3,-1.7) -- (-1,-1) -- (-1,0) -- (-3,1) -- (-3.8,1+0.8*2/5);


    \node[fill=blue,circle,inner sep=1.2pt] at (3,4) {};
    \node[fill=blue,circle,inner sep=1.2pt] at (1,1) {};
    \node[fill=blue,circle,inner sep=1.2pt] at (1,0) {};
    \node[fill=blue,circle,inner sep=1.2pt] at (3,-1) {};
    \node[fill=blue,circle,inner sep=1.2pt] at (2,3) {};
    \node[fill=blue,circle,inner sep=1.2pt] at (0,1) {};
    \node[fill=blue,circle,inner sep=1.2pt] at (-2,1) {};

    \node[fill=blue,circle,inner sep=1.2pt] at (-1,-1) {};
    \node[fill=blue,circle,inner sep=1.2pt] at (-1,0) {};
    \node[fill=blue,circle,inner sep=1.2pt] at (-3,1) {};
    \node[fill=blue,circle,inner sep=1.2pt] at (0,-1) {};
    \node[fill=blue,circle,inner sep=1.2pt] at (2,-1) {};

    \node[fill=blue,circle,inner sep=0.8pt] at (1,-1) {};
    \node[fill=blue,circle,inner sep=0.8pt] at (-1,1) {};
    \node[fill=blue,circle,inner sep=0.8pt] at (1,2) {};

    \node[right] at (1-0.03,0.5) {$a_0$}; 
    \node[right] at (2-2/13,2.5-3/13) {$a_2$};
    \node[right] at (3+1.2,4+1.2*7/5) {$a_4$};
    \node[right] at (2-0.06,-0.4) {$a_{-2}$};

    \node[above left] at (1.08,2-0.02) {$a_1$};
    \node[left] at (2+1.5*1.07,3+2*1.07) {$a_3$};
    \node[above] at (-1,0.95) {$a_{-1}$};

    \draw[blue] ([shift=({atan(-1/2)}:0.2)]1,0) arc (atan(-1/2):90:0.2);
    \draw[blue] ([shift=({atan(-2/5)}:0.2)]3,-1) arc (atan(-2/5):90+atan(2):0.2);
    \draw[blue] ([shift=(-90:0.2)]1,1) arc (-90:atan(3/2):0.2);
    \draw[blue] ([shift=({-90-atan(2/3)}:0.2)]3,4) arc (-90-atan(2/3):atan(7/5):0.2);
    \draw[blue] ([shift=({atan(4/3)}:0.2)]2,3) arc (atan(4/3):225:0.2);
    \draw[blue] ([shift=(45:0.2)]0,1) arc (45:180:0.2);
    \draw[blue] ([shift=(0:0.2)]-2,1) arc (0:90+atan(3):0.2);

    \node[right] at (3.15,-0.88) {$a_{-3}$};
    \node[right] at (1.17,0) {$a_{-1}$};
    \node[right] at (1.17,1) {$a_1$};
    \node[right] at (3.08,3.75) {$a_3$};
    \node[left] at (2-0.1,3.23) {$a_2$};
    \node[above] at (-0.05,1.18) {$a_0$};
    \node[above] at (-2.02,1.16) {$a_{-2}$};

    \draw (1+0.4,-0.16) node[left,circle,inner sep=7pt]{$\vec v_{-2}$};
    \draw (-0.21,1-0.17) node[right,circle,inner sep=7pt]{$\vec v_{-1}$};
    \draw (-2.21,1-0.17) node[right,circle,inner sep=7pt]{$\vec v_{-3}$};
    \node[left,circle,inner sep=7pt] at (3+0.2,4.1) {$\vec v_2$};
    \node[right,circle,inner sep=7pt] at (2-0.22,3-0.13) {$\vec v_1$};
    \draw (3+0.4,-1.16) node[left,circle,inner sep=7pt]{$\vec v_{-4}$};
    \draw (1+0.32,1.1) node[left,circle,inner sep=10pt]{$\vec v_0$};

    \draw (4.26,1.5) node[right]{$\cK_1$};
    \draw (-0.74,4.5) node[right]{$\cK_2$};
  \end{tikzpicture}
  \caption{Klein polygons and continued fractions} \label{fig:KP_and_CF}
\end{figure}

\begin{proposition} \label{prop:KP_and_CF}
  Let $\alpha$ and $\beta$ satisfy \eqref{eq:nice_alpha_beta}. Let $(\vec v_k)_{k\in\Z}$, $(a_k)_{k\in\Z}$ be defined by \eqref{eq:initiation} and Corollary \ref{cor:naming_vertices}. Then
  \[ \alpha=[a_0;a_1,a_2,\ldots]\quad\text{ and }\quad-1/\beta=[a_{-1};a_{-2},a_{-3},\ldots]. \]
  Moreover, for each $k\geq0$ we have $\vec v_k=(q_k,p_k)$, where $p_k$ and $q_k$ are the numerator and denominator of the $k^\text{th}$ convergent to $\alpha$.
\end{proposition}

\begin{proof}
  Proposition \ref{prop:deep_recursion} with $m=0$ under condition \eqref{eq:initiation} immediately implies that $\vec v_k=(q_k,p_k)$ with coprime $p_k$ and $q_k$ such that
  \[ \frac{p_k}{q_k}=[a_0;a_1,\ldots,a_k]. \]
  Furthermore, the points $\vec v_k$ tend to $\cL_\alpha$ as $k\to\infty$. Particularly, $p_k/q_k\to\alpha$ as $k\to\infty$. Hence
  \[ \alpha=\lim_{k\to\infty}[a_0;a_1,\ldots,a_k]=[a_0;a_1,a_2,\ldots]. \]

  As for the statement concerning the expansion of $-1/\beta$, it follows from the one concerning $\alpha$, since the rotation of the whole construction by $\pi/2$ turns $\cL_\alpha$ and $\cL_\beta$ into $\cL_{-1/\beta}$ and $\cL_{-1/\alpha}$ respectively, and we have $-1/\beta>1$ and $-1<-1/\alpha<0$.
\end{proof}

If $\alpha>\beta$ but \eqref{eq:nice_alpha_beta} does not hold, the correspondence demonstrated in Proposition \ref{prop:KP_and_CF} is no longer exact, because then the segment with the endpoints $(1,0)$ and $([\alpha],0)$ fails to be an edge of the corresponding Klein polygon. One can easily prove the following.

\begin{proposition} \label{prop:nice_alpha_beta_criterion}
  If $\alpha>\beta$, then \eqref{eq:nice_alpha_beta} is equivalent to any of the following statements:

  \textup{(a)} the points $(1,0)$ and $(0,1)$ are vertices of adjacent Klein polygons corresponding to $\cL_\alpha$ and $\cL_\beta$;

  \textup{(b)} the segment with the endpoints $(1,0)$ and $(1,[\alpha])$ is an edge of a Klein polygon corresponding to $\cL_\alpha$ and $\cL_\beta$.
\end{proposition}

Since integer lengths and angles are invariant under the action of operators from $\Gl_2(\Z)$, for arbitrary $\alpha$ and $\beta$ the exact correspondence is restored as follows.

\begin{proposition} \label{prop:KP_and_CF_arbitrary_case}
  Let $\cK_1$ and $\cK_2$ be adjacent Klein polygons corresponding to $\cL_\alpha$ and $\cL_\beta$ separated by $\cL_\alpha$. Let $(\vec v_k)_{k\in\Z}$, $(a_k)_{k\in\Z}$ be as in Corollary \ref{cor:naming_vertices}, with arbitrary choice of $\vec v_0$. Then there is a unique operator $A\in\Gl_2(\Z)$ such that

  \textup{(a)} $A\vec v_{-2}=(1,0)$, $A\vec v_{-1}=(0,1)$, $A\vec v_0=(1,a_0)$;

  \textup{(b)} $A\cL_\alpha=\cL_{\alpha'}$, $A\cL_\beta=\cL_{\beta'}$ (particularly, $\alpha\sim\alpha'$, $\beta\sim\beta'$);

  \textup{\hskip0.23mm(c)\hskip0.23mm} $\alpha'>1$, $-1<\beta'<0$;

  \textup{(d)} the sequence $(a_k)$ coincides with the sequence obtained by gluing together the sequences of partial quotients of $\alpha'$ and $-1/\beta'$.
\end{proposition}

The proof is left to the reader.

\subsection{Quadratic irrationalities}

\paragraph{Lagrange's theorem.}

It follows from Proposition \ref{prop:KP_and_CF} and Corollary \ref{cor:naming_vertices} that if the continued fraction of $\alpha$ is (eventually) periodic, then $(1,\alpha)$ is an eigenvector of an operator from $\Sl_2(\Z)$. Indeed, we may assume that $\alpha>1$ 
and complement it with $\beta$ 
satisfying \eqref{eq:nice_alpha_beta}. Let $(\vec v_k)$ be as in Corollary \ref{cor:naming_vertices}. Then, if $t$ is the period length and $s$ is the length of the preperiod, the operator $A$ determined by
\[ A\vec v_s=\vec v_{s+2t},\quad A\vec v_{s+1}=\vec v_{s+1+2t} \]
maps the whole sequence $(\vec v_k)_{k\geq s}$ onto its proper subset and preserves orientation. Hence $A\in\Sl_2(\Z)$ and $A\cL_\alpha=\cL_\alpha$. Thus $\alpha$ appears to be a quadratic irrationality, which is the simplest half of Lagrange's theorem on continued fractions.

In order to prove the hardest half of Lagrange's theorem, let us consider an operator $A\in\Sl_2(\Z)$ with positive irrational eigenvalues. It has eigenvectors $(1,\alpha)$, $(1,\beta)$, $\alpha\neq\beta$. It is easy to see that $\alpha$ and $\beta$ are conjugate quadratic irrationalities. Denote by $\cC$ the angle between $(1,\alpha)$ and $(1,\beta)$ and by $\cK$ the corresponding Klein polygon,
\[ \cK=\conv(\cC\cap\Z^2\backslash\{\vec 0\}). \]
Then $A(\cC)=\cC$, $A(\cK)=\cK$, $A(\partial\cK)=\partial\cK$, where $\partial\cK$ denotes the boundary of $\cK$. This means that the sequence $(a_k)$ of integer lengths and angles written along $\partial\cK$ is mapped onto itself under the action of $A$. Hence it is periodic. Applying Proposition \ref{prop:KP_and_CF_arbitrary_case} we see that $\alpha$ and $\beta$ have eventually periodic continued fraction expansions.

It remains to show that for any quadratic irrationality $\alpha$ there is an operator $A\in\Sl_2(\Z)$ such that $(1,\alpha)$ is its eigenvector. One of the standard ways to do this is to show that if $\alpha$ is a root of $ax^2+2bx+c$, $a,b,c\in\Z$, $ac<0$, then the quadratic form
\[ f(x,y)=cx^2+2bxy+ay^2 \]
admits a nontrivial $\Sl_2(\Z)$-automorphism with nonnegative entries. Such an automorphism would generate a hyperbolic shift with the axes generated by $(1,\alpha)$, $(1,\beta)$, where $\beta$ is the conjugate of $\alpha$.

\begin{remark*}
  The condition $ac<0$ can be easily satisfied due to statements \textup{(b)} and \textup{(c)} of Proposition \ref{prop:KP_and_CF_arbitrary_case}, for if $\alpha$ and $\beta$ are conjugates, so will be $\alpha'$ and $\beta'$.
\end{remark*}

An automorphism of $f(x,y)$ is found by iterating the substitutions
\[ (x,y)\to(x,x+y)\quad\text{ or }\quad(x,y)\to(x+y,y), \]
of which we choose according to the sign of $f(1,1)$. Initially we have $f(0,1)f(1,0)=ac<0$. So, if for the current $f(x,y)$ we have $f(1,1)f(0,1)<0$, we apply the first substitution, if $f(1,1)f(1,0)<0$, we apply the second one. Those products are nonzero, since $f(x,y)$ is never zero at nonzero integer points. The choice of the substitution is determined uniquely and it preserves the condition
\[ f(0,1)f(1,0)<0. \]
Moreover, this process does not change the discriminant $b^2-ac$, which is positive. This bounds possible values of coefficients. Therefore, since this process is invertible, the initial triple $(a,b,c)$ appears again inevitably. Thus we find a nonidentity $A\in\Sl_2(\Z)$ with nonnegative entries such that
\[ f(x,y)=
   (x\,\ y)
   \begin{pmatrix}
     c & b \\
     b & a
   \end{pmatrix}
   \begin{pmatrix}
     x \\
     y
   \end{pmatrix}=
   (x\,\ y)\tr A
   \begin{pmatrix}
     c & b \\
     b & a
   \end{pmatrix}
   A
   \begin{pmatrix}
     x \\
     y
   \end{pmatrix}. \]
Hence $(1,\alpha)$ and $(1,\beta)$ are eigenvectors of $A$, which completes the proof of Lagrange's theorem.

\paragraph{Galois' theorem.}

Let $\alpha$ and $\beta$ be conjugate quadratic irrationalities. Then, as we have just shown, $(1,\alpha)$, $(1,\beta)$ are eigenvectors of an operator $A\in\Sl_2(\Z)$. Taking into account Proposition \ref{prop:KP_and_CF_arbitrary_case} we see that a period of $\alpha$ is a reversed period of $\beta$. Moreover, if $\alpha$ and $\beta$ satisfy \eqref{eq:nice_alpha_beta}, then by Proposition \ref{prop:KP_and_CF} neither $\alpha$, nor $-1/\beta$ have any preperiod. In addition to that, if $\alpha$ is purely periodic, then statement \textup{(b)} of Proposition \ref{prop:nice_alpha_beta_criterion} holds, so that $\alpha$ and $\beta$ satisfy \eqref{eq:nice_alpha_beta}, whence, again by Proposition \ref{prop:KP_and_CF}, $-1/\beta$ has no preperiod.

This gives us Galois' theorem quoted in the Introduction.

\paragraph{Geometry of quadratic irrationalities.}

While proving Lagrange's and Galois' theorems we have particularly shown the following.

\begin{proposition}
  Let $\alpha$ and $\beta$ be quadratic irrationalities. Then the following statements are equivalent:

  \textup{(a)} $\alpha$ and $\beta$ are conjugates;

  \textup{(b)} $(1,\alpha)$ and $(1,\beta)$ are eigenvectors of an operator from $\Sl_2(\Z)$;

  \textup{(c)} the sequence $(a_k)$ of integer lengths and angles written along the boundary of a Klein polygon corresponding to $\cL_\alpha$ and $\cL_\beta$ is periodic;

  \textup{(d)} there is an $A\in\Gl_2(\Z)$ such that for $\alpha'$ and $\beta'$ determined by $A\cL_\alpha=\cL_{\alpha'}$, $A\cL_\beta=\cL_{\beta'}$ we have
  \[ \alpha'=[\overline{a_0;a_1,\ldots,a_t}],\quad-1/\beta'=[\overline{a_t;a_{t-1},\ldots,a_0}]. \]
\end{proposition}

\section{Proof of Theorem \ref{t:cyclic_palindrome_criterion}} \label{sec:proof}

Let $\alpha$ be a quadratic irrationality, and let $(a_k)_{k\in\Z}$ be the sequence of integer lengths and angles written along the boundaries of the Klein polygons corresponding to $\cL_\alpha$ and $\cL_{\bar\alpha}$. Given a specific edge or a vertex of one of the Klein polygons, we may assume that $a_0$ is attached to it.

If $\omega\sim\alpha$, then there is an $A\in\Gl_2(\Z)$ such that $A\cL_\alpha=\cL_\omega$ and $A\cL_{\bar\alpha}=\cL_{\bar\omega}$, so, we have the same sequence $(a_k)$ written along the boundaries of the Klein polygons corresponding to $\cL_\omega$ and $\cL_{\bar\omega}$.

The sequence $(a_k)$ is periodic and its period coincides with the period of the continued fraction of $\alpha$. This period is cyclic palindromic if and only if $(a_k)$ is symmetric. A center of this symmetry is either an element of $(a_k)$, or a space between neighbouring elements. In case it is an element of $(a_k)$, being a positive integer, it is either even, or odd. Thus, a center may be \emph{even}, \emph{odd}, or \emph{intermediate}.

We split the proof of Theorem \ref{t:cyclic_palindrome_criterion} into four lemmas, which are proved similarly.

\begin{lemma} \label{l:even_center}
  The sequence $(a_k)$ has an even center if and only if
  \begin{equation} \label{eq:even_center}
    \alpha\sim\omega:\ \omega+\bar\omega=0.
  \end{equation}
\end{lemma}

\begin{proof}
  Suppose $\omega$ satisfies \eqref{eq:even_center}. Then $\cL_\omega$ and $\cL_{\bar\omega}$ are symmetric with respect to the coordinate axes. Denote by $\cK$ the Klein polygon containing the point $(1,1)$ and set
  \[ A=
     \begin{pmatrix}
       1 & \phantom{-}0 \\
       0 & -1
     \end{pmatrix},\quad
     B=
     \begin{pmatrix}
       -1 & 0 \\
       \phantom{-}0 & 1
     \end{pmatrix}. \]
  Then, if $|\omega|>1$, we have $A\cK=\cK$, and we can attach $a_0$ to the vertical edge of $\cK$ (see Fig. \ref{fig:even_center}). If $|\omega|<1$, we have $B\cK=\cK$, and we can attach $a_0$ to the horizontal edge of $\cK$ (see Fig. \ref{fig:even_center}). Then $a_0$ is even and is a center of $(a_k)$.

  Conversely, given $(a_k)$ with an even center at $a_0$, let us set
  \[ \vec v_{-2}=(1,-a_0/2),\quad\vec v_0=(1,a_0/2) \]
  and apply Proposition \ref{prop:korkina}. We get a Klein polygon $\cK$ corresponding to some $\cL_\omega$ and $\cL_{\bar\omega}$. Clearly, $A\vec v_{-2}=\vec v_0$ and $A\vec v_0=\vec v_{-2}$. Applying Proposition \ref{prop:korkina} to $A\vec v_{-2}$ and $A\vec v_0$ we get
  \[ A\cK=\cK, \]
  whence $A\cL_\omega=\cL_{\bar\omega}$ and $A\cL_{\bar\omega}=\cL_\omega$. Thus, $\bar\omega=-\omega$.
\end{proof}

\begin{remark}
  It is clear from the proof of Lemma \ref{l:even_center} that the upper half of the segment $[\vec v_{-2},\vec v_0]$ provides $\omega$ with a preperiod consisting of $a_0/2$, which results in \eqref{eq:sqrt_expansion}.
\end{remark}

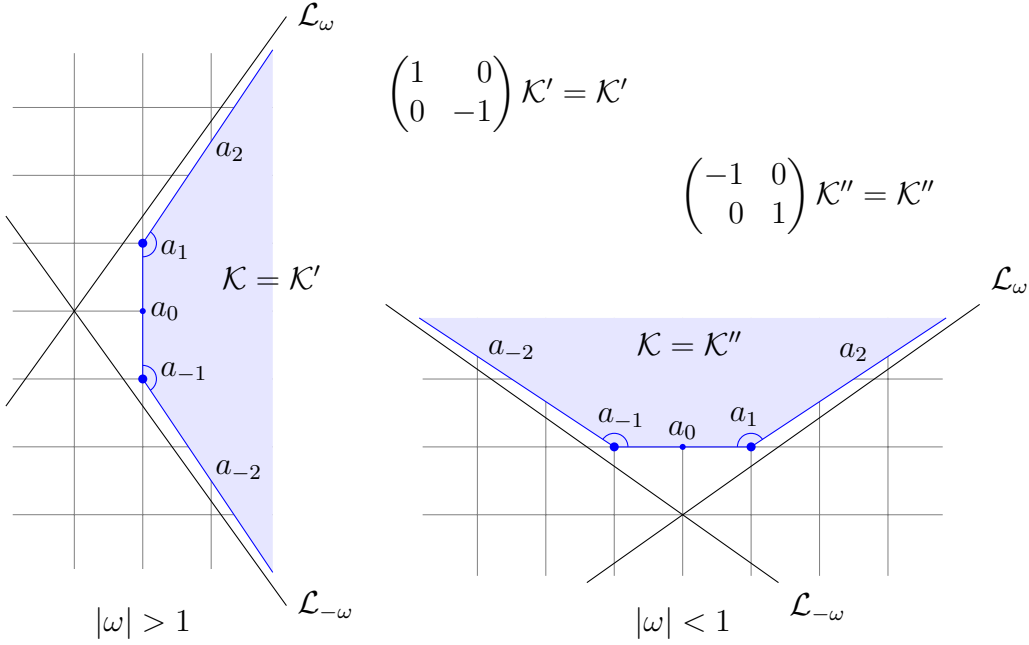
\begin{figure}[h]
  \centering
  \begin{tikzpicture}[scale=0.9]
    \draw[very thin,color=gray,scale=1] (-0.9,-3.8) grid (2.9,3.8);

    \draw[color=black] plot[domain=-1:3.1] (\x, {7*\x/5}) node[right]{$\cL_\omega$};
    \draw[color=black] plot[domain=-1:3.1] (\x, {-7*\x/5}) node[right]{$\cL_{-\omega}$};

    \fill[blue!10!,path fading=east]
        (1+1.9,1+1.9*3/2) -- (1,1) -- (1,-1) -- (1+1.9,-1-1.9*3/2) -- cycle;

    \draw[color=blue] (1+1.9,1+1.9*3/2) -- (1,1) -- (1,-1) -- (1+1.9,-1-1.9*3/2);

    \node[fill=blue,circle,inner sep=1.2pt] at (1,1) {};
    \node[fill=blue,circle,inner sep=1.2pt] at (1,-1) {};
    \node[fill=blue,circle,inner sep=0.8pt] at (1,0) {};

    \draw[blue] ([shift=(-90:0.2)]1,1) arc (-90:atan(3/2):0.2);
    \draw[blue] ([shift=({-atan(3/2)}:0.2)]1,-1) arc (-atan(3/2):90:0.2);

    \node[right] at (0.95,0) {$a_0$};
    \node[right] at (1+0.1,0.9) {$a_1$};
    \node[right] at (1+0.1,-0.9) {$a_{-1}$};
    \node[right] at (1.9,2.35) {$a_2$};
    \node[right] at (1.9,-2.35) {$a_{-2}$};

    \draw (2,0.5) node[right]{$\cK=\cK'$};
    \draw (1,-4.2) node[below]{$|\omega|>1$};
  \end{tikzpicture}
  \begin{tikzpicture}[scale=0.9]
    \draw[very thin,color=gray,scale=1] (-3.8,-0.9) grid (3.8,2.9);

    \draw[color=black] plot[domain=-7/5:3.1*7/5] (\x, {5*\x/7}) node[above right]{$\cL_\omega$};
    \draw[color=black] plot[domain=-3.1*7/5:7/5] (\x, {-5*\x/7}) node[below right]{$\cL_{-\omega}$};

    \fill[blue!10!,path fading=north]
        (1+1.9*3/2,1+1.9) -- (1,1) -- (-1,1) -- (-1-1.9*3/2,1+1.9) -- cycle;

    \draw[color=blue] (1+1.9*3/2,1+1.9) -- (1,1) -- (-1,1) -- (-1-1.9*3/2,1+1.9);

    \node[fill=blue,circle,inner sep=1.2pt] at (1,1) {};
    \node[fill=blue,circle,inner sep=1.2pt] at (-1,1) {};
    \node[fill=blue,circle,inner sep=0.8pt] at (0,1) {};

    \draw[blue] ([shift=({atan(2/3)}:0.2)]1,1) arc (atan(2/3):180:0.2);
    \draw[blue] ([shift=(0:0.2)]-1,1) arc (0:180-atan(2/3):0.2);

    \node[above] at (0,0.95) {$a_0$};
    \node[above] at (0.9,1+0.15) {$a_1$};
    \node[above] at (-0.9,1+0.1) {$a_{-1}$};
    \node[above] at (2.5,2.1) {$a_2$};
    \node[above] at (-2.5,2.1) {$a_{-2}$};

    \draw (0.1,2.2) node[above]{$\cK=\cK''$};
    \draw (0,-1.2) node[below]{$|\omega|<1$};

    \draw (-4.5,6.2) node[right]
        {$\begin{pmatrix}
            1 & \phantom{-}0 \\
            0 & -1
          \end{pmatrix}
          \cK'=\cK'$};
    \draw (-0.2,4.7) node[right]
        {$\begin{pmatrix}
            -1 & 0 \\
            \phantom{-}0 & 1
          \end{pmatrix}
          \cK''=\cK''$};
  \end{tikzpicture}
  \caption{Symmetries for the case $\omega+\bar\omega=0$} \label{fig:even_center}
\end{figure}

\begin{lemma} \label{l:odd_center_1/2}
  The sequence $(a_k)$ has an odd center if and only if
  \begin{equation} \label{eq:odd_center_1/2}
    \alpha\sim\omega:\ \omega+\bar\omega=1.
  \end{equation}
\end{lemma}

\begin{proof}
  Suppose $\omega$ satisfies \eqref{eq:odd_center_1/2}. Then $\cL_\omega$ and $\cL_{\bar\omega}$ are interchanged by
  \[ A=
     \begin{pmatrix}
       1 & \phantom{-}0 \\
       1 & -1
     \end{pmatrix}\quad\text{ and }\quad
     B=
     \begin{pmatrix}
       -1 & 0 \\
       -1 & 1
     \end{pmatrix}. \]
  Denote by $\cK$ the Klein polygon containing the point $(1,1)$.

  We may assume that $\omega>\bar\omega$. Particularly, $\omega>1/2$.

  If $\omega>1$, there is an edge of $\cK$ which contains $(1,0)$ and $(1,1)$. Let us attach $a_0$ to this edge (see Fig. \ref{fig:odd_center_1/2}). Then $a_0$ is odd and is a center of $(a_k)$, since
  \[ A\cK=\cK\quad\text{ and }\quad A
     \begin{pmatrix}
       1 \\
       1/2
     \end{pmatrix}=
     \begin{pmatrix}
       1 \\
       1/2
     \end{pmatrix}. \]

  If $\omega<1$, there is an edge of $\cK$ which contains $(-1,0)$ and $(1,1)$, and we attach $a_0$ to this edge (see Fig. \ref{fig:odd_center_1/2}). Then, again, $a_0$ is odd and is a center of $(a_k)$, since
  \[ B\cK=\cK\quad\text{ and }\quad B
     \begin{pmatrix}
       0 \\
       1/2
     \end{pmatrix}=
     \begin{pmatrix}
       0 \\
       1/2
     \end{pmatrix}. \]

  Conversely, given $(a_k)$ with an odd center at $a_0$, let us set
  \[ \vec v_{-2}=(1,(1-a_0)/2),\quad\vec v_0=(1,(1+a_0)/2) \]
  and apply Proposition \ref{prop:korkina}. We get a Klein polygon $\cK$ corresponding to some $\cL_\omega$ and $\cL_{\bar\omega}$. Clearly, $A\vec v_{-2}=\vec v_0$ and $A\vec v_0=\vec v_{-2}$. Same as in the proof of Lemma \ref{l:even_center}, we apply again Proposition \ref{prop:korkina} to $A\vec v_{-2}$ and $A\vec v_0$ and get
  \[ A\cK=\cK, \]
  whence $A\cL_\omega=\cL_{\bar\omega}$ and $A\cL_{\bar\omega}=\cL_\omega$. Thus, $\bar\omega=1-\omega$.
\end{proof}

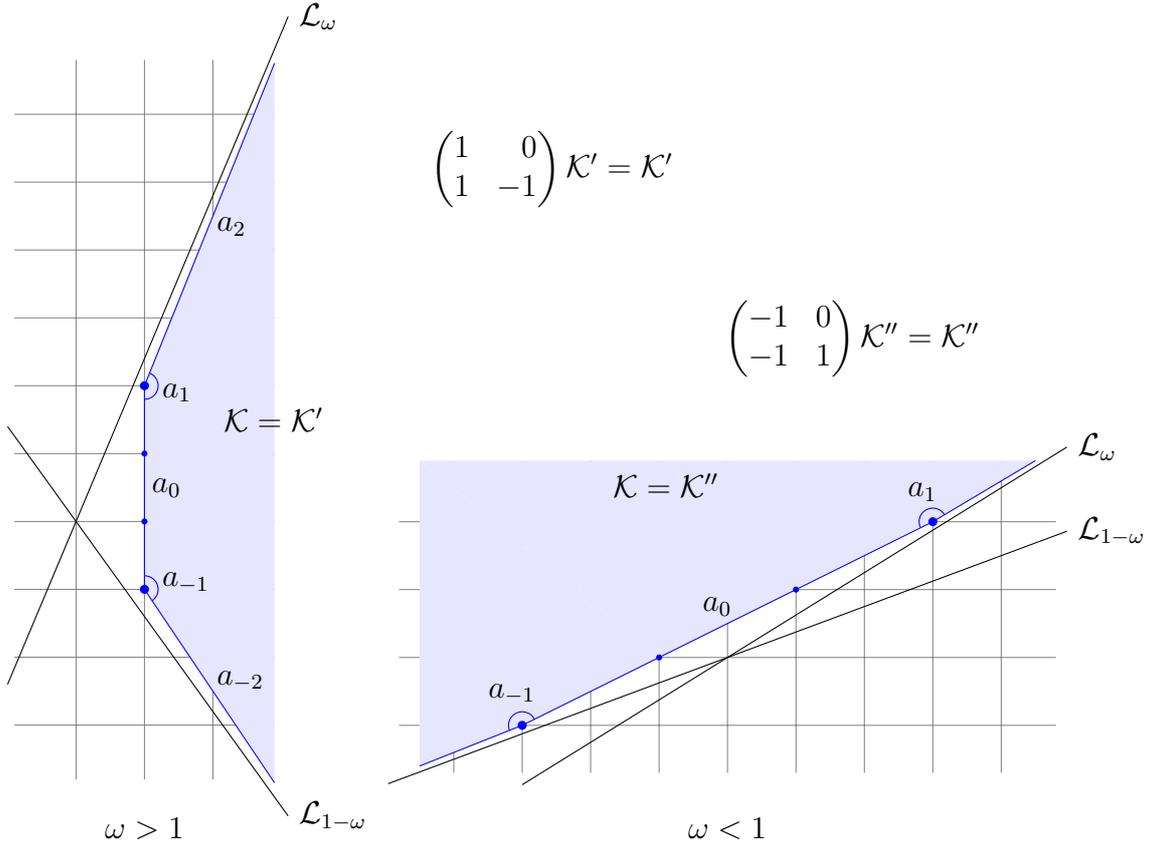
\begin{figure}[h]
  \centering
  \begin{tikzpicture}[scale=0.9]
    \draw[very thin,color=gray,scale=1] (-0.9,-3.8) grid (2.9,6.8);

    \draw[color=black] plot[domain=-1:3.1] (\x, {12*\x/5}) node[right]{$\cL_\omega$};
    \draw[color=black] plot[domain=-1:3.1] (\x, {-7*\x/5}) node[right]{$\cL_{1-\omega}$};

    \fill[blue!10!,path fading=east]
        (1+1.9,2+1.9*5/2) -- (1,2) -- (1,-1) -- (1+1.9,-1-1.9*3/2) -- cycle;

    \draw[color=blue] (1+1.9,2+1.9*5/2) -- (1,2) -- (1,-1) -- (1+1.9,-1-1.9*3/2);

    \node[fill=blue,circle,inner sep=1.2pt] at (1,2) {};
    \node[fill=blue,circle,inner sep=1.2pt] at (1,-1) {};
    \node[fill=blue,circle,inner sep=0.8pt] at (1,1) {};
    \node[fill=blue,circle,inner sep=0.8pt] at (1,0) {};

    \draw[blue] ([shift=(-90:0.2)]1,2) arc (-90:atan(5/2):0.2);
    \draw[blue] ([shift=({-atan(3/2)}:0.2)]1,-1) arc (-atan(3/2):90:0.2);

    \node[right] at (0.95,0.5) {$a_0$};
    \node[right] at (1+0.1,1.9) {$a_1$};
    \node[right] at (1+0.1,-0.9) {$a_{-1}$};
    \node[right] at (1.9,4.35) {$a_2$};
    \node[right] at (1.9,-2.35) {$a_{-2}$};

    \draw (2,1.5) node[right]{$\cK=\cK'$};
    \draw (1,-4.2) node[below]{$\omega>1$};
  \end{tikzpicture}
  \begin{tikzpicture}[scale=0.9]
    \draw[very thin,color=gray,scale=1] (-4.8,-1.7) grid (4.8,2.9);

    \draw[color=black] plot[domain=-3:3.1*8/5] (\x, {5*\x/8}) node[right]{$\cL_\omega$};
    \draw[color=black] plot[domain=-3.1*8/5:3.1*8/5] (\x, {3*\x/8}) node[right]{$\cL_{1-\omega}$};

    \fill[blue!10!,path fading=north]
        (3+0.9*5/3,2+0.9) -- (3,2) -- (1,1) -- (-1,0) -- (-3-0.9*5/3,2+0.9) -- cycle;
    \fill[blue!10!,path fading=west]
        (-3-0.6*5/2,-1-0.6) -- (-3,-1) -- (-1,0) -- (-3-0.9*5/3,2+0.9) -- cycle;

    \draw[color=blue] (3+0.9*5/3,2+0.9) -- (3,2) -- (-3,-1) -- (-3-0.6*5/2,-1-0.6);

    \node[fill=blue,circle,inner sep=1.2pt] at (3,2) {};
    \node[fill=blue,circle,inner sep=1.2pt] at (-3,-1) {};
    \node[fill=blue,circle,inner sep=0.8pt] at (1,1) {};
    \node[fill=blue,circle,inner sep=0.8pt] at (-1,0) {};

    \draw[blue] ([shift=({atan(3/5)}:0.2)]3,2) arc (atan(3/5):180+atan(1/2):0.2);
    \draw[blue] ([shift=({atan(1/2)}:0.2)]-3,-1) arc (atan(1/2):180+atan(2/5):0.2);

    \node[above] at (-2+1.85,-1+1.45) {$a_0$};
    \node[above] at (2.85,2+0.15) {$a_1$};
    \node[above] at (-4+0.85,-1+0.15) {$a_{-1}$};

    \draw (-1+0.1,2.2) node[above]{$\cK=\cK''$};
    \draw (0,-2.2) node[below]{$\omega<1$};

    \draw (-4.5,7.2) node[right]
        {$\begin{pmatrix}
            1 & \phantom{-}0 \\
            1 & -1
          \end{pmatrix}
          \cK'=\cK'$};
    \draw (-0.2,4.7) node[right]
        {$\begin{pmatrix}
            -1 & 0 \\
            -1 & 1
          \end{pmatrix}
          \cK''=\cK''$};
  \end{tikzpicture}
  \caption{Symmetries for the case $\omega+\bar\omega=1$} \label{fig:odd_center_1/2}
\end{figure}

\begin{lemma} \label{l:odd_center_norm_1}
  The sequence $(a_k)$ has an odd center if and only if
  \begin{equation} \label{eq:odd_center_norm_1}
    \alpha\sim\omega:\ \omega\bar\omega=1.
  \end{equation}
\end{lemma}

\begin{proof}
  Suppose $\omega$ satisfies \eqref{eq:odd_center_norm_1}. Then $\cL_\omega$ and $\cL_{\bar\omega}$ are interchanged by
  \[ A=
     \begin{pmatrix}
       0 & 1 \\
       1 & 0
     \end{pmatrix}. \]
  We may assume that $\omega<0$. Then the points $(1,0)$ and $(0,1)$ lie on an edge of a Klein polygon $\cK$ (see Fig. \ref{fig:odd_center_norm_1}). Attaching $a_0$ to this edge we see that $a_0$ is odd and is a center of $(a_k)$, since
  \[ A\cK=\cK\quad\text{ and }\quad A
     \begin{pmatrix}
       1/2 \\
       1/2
     \end{pmatrix}=
     \begin{pmatrix}
       1/2 \\
       1/2
     \end{pmatrix}. \]

  Conversely, given $(a_k)$ with an odd center at $a_0$, let us set
  \[ \vec v_{-2}=\bigg(\frac{1+a_0}{2},\frac{1-a_0}{2}\bigg),\quad\vec v_0=\bigg(\frac{1-a_0}{2},\frac{1+a_0}{2}\bigg), \]
  notice that $A\vec v_{-2}=\vec v_0$, $A\vec v_0=\vec v_{-2}$, and apply Proposition \ref{prop:korkina} twice. As before, we thus obtain a Klein polygon $\cK$ corresponding to some $\cL_\omega$ and $\cL_{\bar\omega}$ such that $A\cK=\cK$, $A\cL_\omega=\cL_{\bar\omega}$, and $A\cL_{\bar\omega}=\cL_\omega$. Hence $\bar\omega=1/\omega$.
\end{proof}

\begin{figure}[h]
  \centering
  \begin{tikzpicture}[scale=0.9]
    \draw[very thin,color=gray,scale=1] (-2.8,-2.8) grid (4.8,4.8);

    \draw[color=black] plot[domain=-3:5.1] (\x, {-5*\x/8}) node[right]{$\cL_\omega$};
    \draw[color=black] plot[domain=-3.1:2] (\x, {-8*\x/5}) node[right]{$\cL_{1/\omega}$};

    \fill[blue!10!,path fading=north]
        (4.8,4.8) -- (1/2,1/2) -- (-1,2) -- (-1-2.8*2/3,2+2.8) -- cycle;
    \fill[blue!10!,path fading=east]
        (2+2.8,-1-2.8*2/3) -- (2,-1) -- (1/2,1/2) -- (4.8,4.8) -- cycle;

    \draw[color=blue] (2+2.8,-1-2.8*2/3) -- (2,-1) -- (-1,2) -- (-1-2.8*2/3,2+2.8);

    \node[fill=blue,circle,inner sep=1.2pt] at (2,-1) {};
    \node[fill=blue,circle,inner sep=1.2pt] at (-1,2) {};
    \node[fill=blue,circle,inner sep=0.8pt] at (1,0) {};
    \node[fill=blue,circle,inner sep=0.8pt] at (0,1) {};

    \draw[blue] ([shift=({-atan(2/3)}:0.2)]2,-1) arc (-atan(2/3):135:0.2);
    \draw[blue] ([shift=(-45:0.2)]-1,2) arc (-45:90+atan(2/3):0.2);

    \node[above right] at (0.4,0.4) {$a_0$};
    \node[above right] at (-1,2) {$a_1$};
    \node[above right] at (2,-1) {$a_{-1}$};
    \node[above right] at (3.4,-2.1) {$a_{-2}$};
    \node[above right] at (-2.05,3.4) {$a_2$};

    \draw (2.17,2.17) node[above right]{$\cK$};

    \draw (6,1) node[right]
        {$\begin{pmatrix}
            0 & 1 \\
            1 & 0
          \end{pmatrix}
          \cK=\cK$};
  \end{tikzpicture}
  \caption{Symmetry for the case $\omega\bar\omega=1$} \label{fig:odd_center_norm_1}
\end{figure}
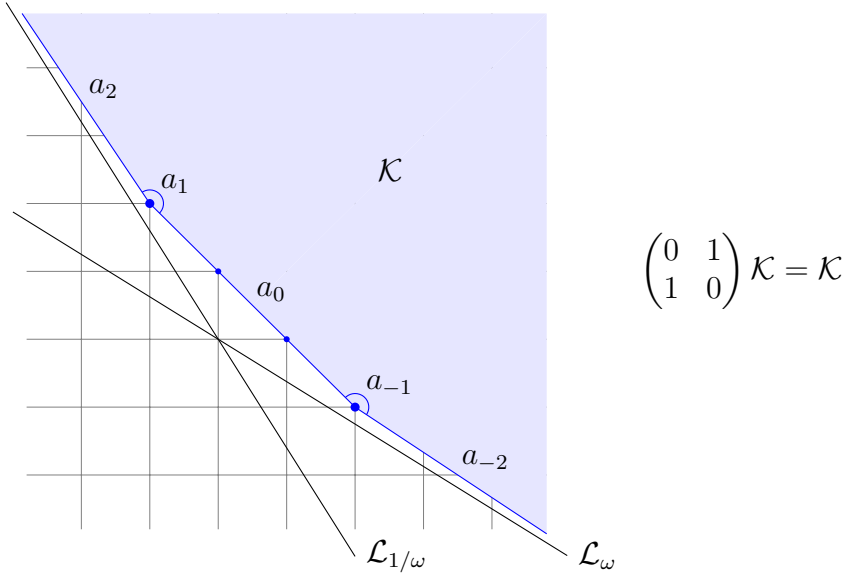

\begin{lemma} \label{l:intermediate_center}
  The sequence $(a_k)$ has an intermediate center if and only if
  \begin{equation} \label{eq:intermediate_center}
    \alpha\sim\omega:\ \omega\bar\omega=-1.
  \end{equation}
\end{lemma}

\begin{proof}
  Suppose $\omega$ satisfies \eqref{eq:intermediate_center}. Then $\cL_\omega$ and $\cL_{\bar\omega}$ are orthogonal and are interchanged by
  \[ A=
     \begin{pmatrix}
       0 & -1 \\
       1 & \phantom{-}0
     \end{pmatrix}. \]
  We may assume that $\omega>1$. Then $\omega$ and $\bar\omega$ satisfy \eqref{eq:nice_alpha_beta} and we can set
  \[ \vec v_{-2}=(1,0),\ \ \vec v_{-1}=(0,1), \]
  attaching thus $a_0$ to the vertical edge of the Klein polygon corresponding to the angle between $(1,\omega)$ and $(1,\bar\omega)$ (see Fig. \ref{fig:intermediate_center}). Then for each even $k$ we have $A\vec v_k=\vec v_{-3-k}$\,, whence it follows that for each integer $k$ we have
  \begin{equation} \label{eq:intermediate_symmetry}
    a_k=a_{-1-k}\,,
  \end{equation}
  i.e. the space between $a_0$ and $a_{-1}$ is a center of symmetry of $(a_k)$.

  Conversely, given $(a_k)$ with an intermediate center between $a_0$ and $a_{-1}$, let us set
  \[ \vec v_{-2}=(1,0),\quad\vec v_0=(1,a_0). \]
  Applying Proposition \ref{prop:korkina} we get a Klein polygon $\cK_1$ with vertices $\vec v_{2m}$, $m\in\Z$. The points $\vec v_{2m+1}$, $m\in\Z$, are vertices of an adjacent Klein polygon $\cK_2$. Since $(a_k)$ satisfies \eqref{eq:intermediate_symmetry}, 1-skeletons of $\cK_1$ and $\cK_2$ are isomorphic, so that by Corollary \ref{cor:korkina} there is a $B\in\Gl_2(\Z)$ such that
  \[ B\vec v_k=\vec v_{-3-k}\,. \]
  But the segment $[\vec v_{-1},\vec v_{-3}]$, being parallel to $\vec v_{-2}$, is orthogonal to $[\vec v_{-2},\vec v_0]$. Hence $B=A$ and, thus, $\omega\bar\omega=-1$.
\end{proof}

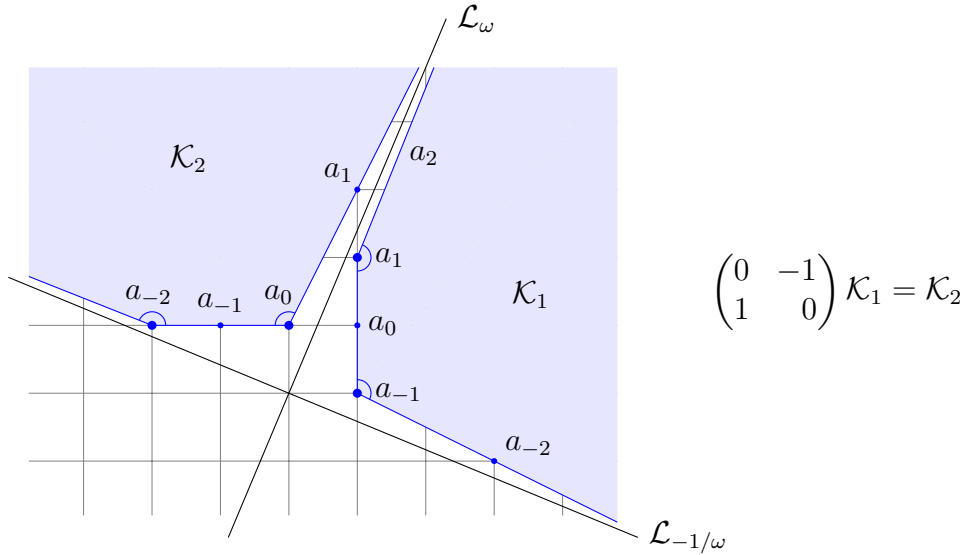
\begin{figure}[h]
  \centering
  \begin{tikzpicture}[scale=0.9]
    \draw[very thin,color=gray,scale=1] (-3.8,-1.8) grid (4.8,4.8);

    \draw[color=black] plot[domain=-5.1*25/144:2.3] (\x, {12*\x/5}) node[right]{$\cL_\omega$};
    \draw[color=black] plot[domain=-4.1:5.1] (\x, {-5*\x/12}) node[right]{$\cL_{-1/\omega}$};

    \fill[blue!10!,path fading=north]
        (1+3.8,2+2.8) -- (1,0) -- (1,2) -- (1+2.8*2/5,2+2.8) -- cycle;
    \fill[blue!10!,path fading=east]
        (1+3.8,-3.8/2) -- (1,0) -- (1+3.8,2+2.8) -- cycle;
    \fill[blue!10!,path fading=north]
        (3.8/2,1+3.8) -- (0,1) -- (-2-1.8,1+3.8) -- cycle;
    \fill[blue!10!,path fading=west]
        (-2-1.8,1+3.8) -- (0,1) -- (-2,1) -- (-2-1.8,1+1.8*2/5) -- cycle;

    \draw[color=blue] (1+3.8,-3.8/2) -- (1,0) -- (1,2) -- (1+2.8*2/5,2+2.8);
    \draw[color=blue] (3.8/2,1+3.8) -- (0,1) -- (-2,1) -- (-2-1.8,1+1.8*2/5);

    \node[fill=blue,circle,inner sep=1.2pt] at (1,0) {};
    \node[fill=blue,circle,inner sep=1.2pt] at (1,2) {};
    \node[fill=blue,circle,inner sep=1.2pt] at (0,1) {};
    \node[fill=blue,circle,inner sep=1.2pt] at (-2,1) {};
    \node[fill=blue,circle,inner sep=0.8pt] at (1,1) {};
    \node[fill=blue,circle,inner sep=0.8pt] at (-1,1) {};
    \node[fill=blue,circle,inner sep=0.8pt] at (3,-1) {};
    \node[fill=blue,circle,inner sep=0.8pt] at (1,3) {};

    \draw[blue] ([shift=({-atan(1/2)}:0.2)]1,0) arc (-atan(1/2):90:0.2);
    \draw[blue] ([shift=(-90:0.2)]1,2) arc (-90:atan(5/2):0.2);
    \draw[blue] ([shift=({atan(2)}:0.2)]0,1) arc (atan(2):180:0.2);
    \draw[blue] ([shift=(0:0.2)]-2,1) arc (0:90+atan(5/2):0.2);

    \node[right] at (1,1) {$a_0$};
    \node[right] at (1.1,2) {$a_1$};
    \node[right] at (1.1,0) {$a_{-1}$};
    \node[above right] at (3,-1.1) {$a_{-2}$};
    \node[right] at (1.6,3.5) {$a_2$};
    \node[above] at (-0.15,1.15) {$a_0$};
    \node[left] at (1.1,3.25) {$a_1$};
    \node[above] at (-1,1) {$a_{-1}$};
    \node[above] at (-2.05,1.12) {$a_{-2}$};

    \draw (3.1,1.1) node[above right]{$\cK_1$};
    \draw (-2+0.1,3.1) node[above right]{$\cK_2$};

    \draw (6,1.5) node[right]
        {$\begin{pmatrix}
            0 & -1 \\
            1 & \phantom{-}0
          \end{pmatrix}
          \cK_1=\cK_2$};
  \end{tikzpicture}
  \caption{Symmetry for the case $\omega\bar\omega=-1$} \label{fig:intermediate_center}
\end{figure}

Combining Lemmas \ref{l:even_center}--\ref{l:intermediate_center} we get Theorem \ref{t:cyclic_palindrome_criterion}.

\end{document}